\documentclass{amsart}

\usepackage{mathrsfs}
\usepackage{amsfonts}
\usepackage{amsmath}
\usepackage{amssymb}
\usepackage{mathtools}
\usepackage{float}
\usepackage{xcolor}
\usepackage{enumerate}
\usepackage{tikz}
\usepackage{tikz-cd}
\usetikzlibrary{arrows,automata}

\makeatletter
\@namedef{subjclassname@2020}{%
  \textup{2020} Mathematics Subject Classification}
\makeatother

\DeclareFontFamily{U}{wncy}{}
    \DeclareFontShape{U}{wncy}{m}{n}{<->wncyr10}{}
    \DeclareSymbolFont{mcy}{U}{wncy}{m}{n}
    \DeclareMathSymbol{\sh}{\mathord}{mcy}{"78}

\newtheorem{theoremB}{Theorem}

\newtheorem{thm}{Theorem}[section]
\newtheorem{prop}[thm]{Proposition}
\newtheorem{lem}[thm]{Lemma}
\newtheorem{cor}[thm]{Corollary}

\theoremstyle{definition}

\newtheorem{question}[thm]{Question}
\newcommand{\pres}[3]{\textnormal{#1} \langle #2 \mid #3 \rangle}

\newcommand{\lra}[1]{\xleftrightarrow{}^\ast_{#1}}
\newcommand{\xra}[1]{\xrightarrow{}^\ast_{#1}}
\newcommand{\xr}[1]{\xrightarrow{}_{#1}}

\newcommand{\trev}{\text{rev}}

\DeclareMathOperator{\WP}{WP}

\DeclareMathOperator{\Pre}{Pre}
\DeclareMathOperator{\Suf}{Suf}

\author{Carl-Fredrik Nyberg-Brodda}
\address{Department of Mathematics, Alan Turing Building, University of Manchester, UK.}
\email{carl-fredrik.nybergbrodda@manchester.ac.uk}
\thanks{The author gratefully acknowledges funding from the Dame Kathleen Ollerenshaw Trust, which is supporting his research at the University of Manchester.}

\title[Maximal Subgroups of Context-Free Monoids]{Non-finitely Generated Maximal Subgroups of Context-free Monoids}

\subjclass[2020]{20F10 (primary) 08A50, 20F05, 68R15, 68Q42}

\date{\today}

\begin{document}

\begin{abstract}
A finitely generated group or monoid is said to be context-free if it has context-free word problem. In this note, we give an example of a context-free monoid, none of whose maximal subgroups are finitely generated. This answers a question of Brough, Cain \& Pfeiffer on whether the group of units of a context-free monoid is always finitely generated, and highlights some of the contrasts between context-free monoids and context-free groups. Finally, we ask whether the group of units of a context-free monoid is always coherent.
\end{abstract}

\maketitle

\section{Introduction}\label{Sec: Introduction}

\noindent The study of language-theoretic properties of algebraic structures can be seen as a natural extension of the work by Thue \cite{Thue1914} in the early 20th century on the word problem (as a decision problem) for monoids. Thue, working with what today are called \textit{rewriting systems} (or \textit{semi-Thue systems}) achieved a remarkable number of results that are today easily recognisable; for example, he formulated the word problem for finitely presented monoids, and solved this problem for certain one-relation monoids. This can be seen as a first step in exploring the connection between rewriting of words -- and, more generally, formal language theory -- and algebraic structures. An{\={\i}}s{\={\i}}mov \cite{Anisimov1971} significantly strengthened this connection in 1971. For a finitely generated group $G$, he studied the set of words over a given finite generating set equal to the identity element in $G$, calling it the \textit{word problem} of $G$. He showed that this language is regular if and only if $G$ is finite. This was extended in 1983 by the remarkable theorem of Muller \& Schupp \cite{Muller1983}, who proved (when supplemented by a result of Dunwoody's \cite{Dunwoody1985}) that the word problem for $G$ is a context-free language if and only if $G$ is virtually free, i.e. $G$ has a free subgroup of finite index. The study of the word problem of groups for other classes of languages has expanded in many directions since, see. e.g. \cite{Holt2005,Brough2014,Gilman2018}.

Inspired by this success, Duncan \& Gilman \cite[\S5]{Duncan2004} introduced a generalisation of the word problem to finitely generated monoids. This generalisation is not a direct translation; indeed, the language of words equal to the identity element in a monoid is, generally speaking, completely insufficient to describe the structure of the monoid (for notable exceptions to this, see \cite{NybergBrodda2020b}). However, their generalisation, which we shall describe in \S\ref{Sec: Background}, has the same language-theoretic properties as the original definition in the case that the monoid in question is a group; furthermore, its properties does not generally depend on the finite generating set chosen for the monoid. This generalisation also has the pleasing property that a monoid has regular word problem if and only if it is finite; that is, the analogue for monoids of An{\={\i}}s{\={\i}}mov's theorem holds.

However, the structure of monoids with context-free word problem seems at present hopelessly difficult to classify in full generality. Any theorem along the lines of the Muller-Schupp theorem for monoids rather than groups would have to be exceptionally deep indeed. One initial hurdle is that there is not even a natural notion of what a ``finite index submonoid'' of a monoid should be; or rather, perhaps, there are too many, none of which is quite satisfactory for these purposes (see e.g. \cite{Jura1978,Ruskuc1998,Gray2008,Gray2011c}). Even answering the question of what a ``virtually free'' monoid should be is hence already a non-trivial task. On the other hand, some less ambitious facts about context-free monoids are known; for example, one can easily show that free monoids have context-free word problem, and recently success has been found regarding closure properties of the class of context-free monoids \cite{Brough2019}.

The study of semigroup and monoid properties inherited from and by their maximal subgroups has a seen a good deal of attention in the literature.\footnote{Indeed, A. K. Sushkevich, one of the fathers of semigroup theory, considered a problem about semigroups to be solved ``if it could be reduced to a problem about groups'' \cite[\S38]{Sushkevich1922}.} For some examples, see e.g. \cite{Golubov1975,Duncan1990,Ruskuc1999,Gray2011,Gray2011b}. This study has also seen success regarding formal language theory. For example, the author has shown that the language-theoretic properties of the word problem of a finitely presented \textit{special} monoid (see \S\ref{Sec: Background}) are completely determined by its group of units \cite{NybergBrodda2020b}. Here the group of units of a monoid $M$ is the maximal subgroup containing the identity element $1$ of $M$. The author has also proved analogous results for \textit{subspecial} and \textit{compressible} monoids \cite{NybergBrodda2020c}. Thus it is also natural to ask to what extent the word problem (in the above sense) of a monoid relates to the word problems of its maximal subgroups. 

In line with this, Brough, Cain \& Pfeiffer posed the question of whether the group of units of a context-free monoid is necessarily finitely generated.\footnote{This appears as Question~10.4 in a pre-print version of \cite{Brough2019}, available at \texttt{arXiv:1903.10493}, but not in the final published version.} In this note, we answer this question negatively. In fact, we show something stronger.

\begin{theoremB}
There exists a context-free monoid none of whose maximal subgroups are finitely generated.
\end{theoremB}

As context-free groups and monoids are, as part of their definition, finitely generated, it follows further that a maximal subgroup of a context-free monoid need not be context-free itself. This highlights the fact that a general reduction of the properties of context-free monoids to those of context-free groups seems difficult. 

We remark that the monoid with the above properties constructed in this note is a non-finitely presented special monoid. Strong structural results are known for finitely presented special monoids via Makanin's work on such monoids \cite{Makanin1966b,Makanin1966}, see also Zhang \cite{Zhang1992}. However, their results are stated only in the finitely presented case, and though much of their analysis carries over to the non-finitely presented case, we have, in the interest of remaining self-contained, opted to give direct proofs of all necessary statements. These proofs should also serve to give a taste of the flavour of many of the arguments arising in the theory of special monoids and combinatorial monoid theory. We remark also that some results for finitely presented special monoids fail in the non-finitely presented case; for example, one can show that a finitely presented special monoid $M$ has context-free word problem if and only if its group of units does \cite{NybergBrodda2020b}. This equivalence no longer holds if the assumption of finite presentability is dropped; in fact, there exists a finitely generated special monoid whose word problem is not context-free, and yet having trivial (and hence context-free) group of units \cite{Thesis}.

\section{Background}\label{Sec: Background}

\noindent We assume the reader is familiar with regular and context-free languages, as well as some elementary properties of codes. For some background on this, and other topics in formal language theory, we refer the reader to standard books on the subject \cite{Harrison1978} \cite{Hopcroft1979}. The paper also assumes familiarity with the basics of the theory of monoid and group presentations, which will be written as $\pres{Mon}{A}{R}$ and $\pres{Gp}{A}{R}$, respectively. For further background information and examples of this theory, see e.g. \cite{Adian1966,Magnus1966},  \cite{Neumann1967,Lyndon1977,Campbell1995}.

We begin by fixing some notation and presenting some basics of rewriting systems. We refer the reader to the monographs by Jantzen \cite{Jantzen1988} and Book \& Otto \cite{Book1993} for a more thorough background on this topic, including its history.

Let $A$ be a finite alphabet, and let $A^\ast$ denote the free monoid on $A$, with identity element denoted $\varepsilon$ or $1$, depending on the context. Let $A^+$ denote the free semigroup on $A$, i.e. $A^+ = A^\ast \setminus \{ \varepsilon\}$. For $u, v \in A^\ast$, by $u \equiv v$ we mean that $u$ and $v$ are the same word. For $w \in A^\ast$, we let $|w|$ denote the \textit{length} of $w$, inductively defined by $|\varepsilon| = 0$ and $|wa| = |w| + 1$ for $w \in A^\ast$ and $a \in A$. We use $\Pre(w)$ (resp.\ $\Suf(w)$) to denote the set of non-empty prefixes (resp.\ non-empty suffixes) of the word $w \in A^\ast$. If two distinct words $u, v \in A^+$ are such that $\Pre(u) \cap \Suf(v) \neq \varnothing$, then we say that $u$ and $v$ \textit{overlap}. We say that $w \in A^+$ \textit{overlaps} with itself if $\Pre(w) \cap \Suf(w) \neq \{ w \}$, and we say that $w$ \textit{self-overlap free} (sometimes also called a \textit{bifix-free word} or \textit{hypersimple word}) if it does not overlap with itself. For example, the word $xyzabcxyz$ is not self-overlap free, whereas \textit{arghmgog}, on the other hand, is.\footnote{The charming word \textit{arghmgog} was invented, and used exactly once, by Turing in the introduction to his paper demonstrating the general undecidability of the word problem for finitely presented cancellative semigroups \cite[p.\ 491]{Turing1950}.} Equality of words $u, v \in A^\ast$ in the monoid $M = \pres{Mon}{A}{R}$ is denoted $u =_M v$. 

Let $G$ be a group with finite (group) generating set $A$, with $A^{-1}$ the set of inverses of the generators $A$. The language 
\[
\{ w \mid w  \in (A \cup A^{-1})^\ast, w =_G 1 \}
\]
is called the (group-theoretic) \textit{word problem} for $G$, see e.g. \cite{Anisimov1971,Muller1983}. Let $M$ be a monoid with a finite generating set $A$. Translating the above definition of the word problem directly to $M$ does not, in general, yield much insight into the structure of $M$. Duncan \& Gilman \cite[p.\ 522]{Duncan2004} instead introduced a different generalisation of the group-theoretic word problem to all monoids. The \textit{word problem of $M$ with respect to $A$} is defined as the language
\[
\{ u \# v^\trev \mid u, v \in A^\ast, u =_M v\},
\]
where $\#$ is some symbol not in $A$, and $v^\trev$ denotes the reversal of $v$, i.e. the word $v$ read backwards. We remark that although Duncan \& Gilman initialised the proper study of this language-theoretic word problem for monoids, the language in the definition was studied already by Book, Jantzen \& Wrathall in 1982, see especially \cite[Corollary~3.8]{Book1982}. This observation does not seem to appear anywhere in the literature on the word problem for monoids. We say that $M$ has \textit{context-free} word problem if this language is context-free; this can be shown to not depend on the finite generating set chosen \cite[Theorem~5.2]{Duncan2004}. Furthermore, a group has context-free word problem in the usual sense if and only if the above language is a context-free language \cite[Theorem~3]{Duncan2004}. If context is clear, we will call a monoid \textit{context-free} if its word problem is context-free. Note that non-finitely generated monoids are not context-free. 

We now give the basic definitions regarding rewriting systems. A \textit{rewriting system} $R$ on $A$ is a subset of $A^\ast \times A^\ast$. An element of $R$ is called a \textit{rule}. The system $R$ induces several relations on $A^\ast$. We will write $u \xr{R} v$ if and only if there exist $x, y \in A^\ast$ and a rule $(\ell, r) \in T$ such that $u \equiv x\ell y$ and $v \equiv xry$. We let $\xra{R}$ denote the reflexive and transitive closure of $\xr{R}$. We denote by $\lra{R}$ the symmetric and transitive closure of $\xr{R}$. The relation $\lra{R}$ defines the least congruence on $A^\ast$ containing $\xra{R}$. The monoid $\pres{Mon}{A}{R}$ will throughout this note, and without comment, be naturally be identified with the quotient $A^\ast / \lra{R}$. 

If $u \in A^\ast$ is such that there is no $u \xr{R} v$, then $u$ is said to be \textit{$R$-irreducible}. We say that $R$ is \textit{terminating} (or \textit{Noetherian}) if there is no infinite chain of the form $u_1 \xr{R} u_2 \xr{R} \cdots$, and we say that it is \textit{confluent} if for all $u, v, w \in A^\ast$, we have that $u \xra{R} v$ and $u \xra{R} w$ together imply that there exists some $z \in A^\ast$ with $v \xra{R} z$ and $w \xra{R} z$. We say that it is \textit{locally confluent} if $u \xr{R} v$ and $u \xr{R} w$ together imply that there exists some $z \in A^\ast$ with $v \xra{R} z$ and $w \xra{R} z$. Newman's classical lemma (see \cite{Newman1942}) says that a terminating and locally confluent rewriting system is confluent. If $R$ is confluent and terminating, we say it is \textit{complete}. A rewriting system $R \subseteq A^\ast \times A^\ast$ is said to be \textit{monadic} if $(u, v) \in R$ implies $|u| > |v|$ and $v \in A \cup \{ \varepsilon \}$. We say that $R$ is \textit{special} if $(u, v) \in R$ implies $|u| > |v|$ and $v \equiv \varepsilon$. The monadic rewriting system $R$ is said to be \textit{context-free} (resp.\ \textit{regular}) if for every $a \in A \cup \{ \varepsilon \}$, the set $\{ u \mid (u, a) \in R \}$ is a context-free (resp.\ regular) language. 

We fix some terminology regarding monoids. Let $M = \pres{Mon}{A}{R}$. An element $m \in M$ is a \textit{left unit} if there exists $m' \in M$ such that $m'm = 1$ in $M$. We define a \textit{right} unit analogously. We say that $m$ is a \textit{unit} if it is both a left and right unit. The subgroup of $M$ consisting of all units is called the \textit{group of units} of $M$, and is denoted $U(M)$. We say that $u \in A^\ast$ is \textit{left} (resp.\ \textit{right}) invertible in $M$ if it represents a left (resp.\ right) unit in $M$; and $u$ is \textit{invertible} in $M$ if it represents a unit. For an idempotent $e \in E(M)$ of $M$, we define the \textit{maximal subgroup of $M$ containing $e$} to be the group of units of $eMe$. The maximal subgroup of $M$ containing $e$ is a subsemigroup of $M$. It is, additionally, a group -- but note that its identity element will be $e$, rather than $1$. For an equivalent definition, the set of maximal subgroups of $M$ is the set of all group $\mathscr{H}$-classes of $M$ \cite[Ch. 2]{Clifford1961}. If $X \subseteq A^\ast$, we denote by $\langle X \rangle_M$ the submonoid of $M$ generated by the elements represented by the words in $X$. Note that $1 \in \langle X \rangle_M$.

Finally, we say that a monoid of the form $M = \pres{Mon}{A}{w_i = 1 \: (i \in I)}$ is \textit{special}, i.e. a monoid is special if the right-hand sides of all defining relations are empty. Such monoids were introduced by Tseitin \cite{Tseitin1958}, and were extensively studied by Adian \cite{Adian1960} and Makanin \cite{Makanin1966,Makanin1966b}. For a more in-depth analysis of special monoids and their history, we refer the reader to \cite{NybergBrodda2020a,NybergBrodda2020b,NybergBrodda2021e,Zhang1992}. The set $\{ w_i \mid i \in I \}$ is called the set of \textit{defining words} of $M$. We say that a non-empty invertible word is \textit{minimal} if none of its non-empty proper prefixes is invertible in $M$. We remark that minimality of $w$ is clearly seen to be equivalent to $w$ having none of its non-empty proper suffixes being invertible. We will use the term \textit{minimal invertible factor} and \textit{minimal factor} interchangeably. Each defining word is invertible in $M$, and can hence be factored uniquely as $w_i \equiv w_{i,1} w_{i,2} \cdots w_{i,\ell}$ into minimal factors. We set $\Lambda(w_i) = \{ w_{i,1}, \dots, w_{i, \ell}\}$, and let $\Lambda = \Lambda(M)$ denote the set $\bigcup_{i \in I} \Lambda(w_i)$ of all minimal words arising from the factorisation of the defining words of $M$. We say that $\Lambda$ is the set of minimal words associated to $M$. It is clear that $\Lambda$ is a biprefix code, i.e. that $\Lambda \cap A^\ast \Lambda = \Lambda \cap \Lambda A^\ast = \varnothing$.

\section{Proof of the main theorem}

\noindent Let $\Pi_2 = \pres{Mon}{a, b, c}{(ab^ic)^2 = 1 \: (i \geq 1)}$. We will throughout this section let $A = \{ a, b, c\}$. We will show that $\Pi_2$ is context-free, but that none of the maximal subgroups of $\Pi_2$ are finitely generated. We shall use the following lemma to provide a direct path to some simple structural results about the group of units of $\Pi_2$. 

\begin{lem}\label{Lem: T is a CF, complete, defining system}
The rewriting system $\mathcal{T}_2 = \{ ( (ab^ic)^2, \varepsilon) \mid i \geq 1\} \subseteq A^\ast \times A^\ast$ is a context-free, complete rewriting system which defines $\Pi_2$. 
\end{lem}
\begin{proof}
Obviously, $\mathcal{T}_2$ defines $\Pi_2$. As $\mathcal{T}_2$ is length-reducing, it is terminating. To show that $\mathcal{T}_2$ is complete, it thus suffices to show that it is locally confluent; but this is immediate, as the only overlap between two left-hand sides of rules from $\mathcal{T}_2$ are of the form $ab^ic(ab^ic)ab^ic$, where $i \geq 1$, and this overlap is immediately resolved. Hence $\mathcal{T}_2$ is complete. Finally, the set of left-hand sides of the empty word in $\mathcal{T}_2$ is the language $\{ ab^icab^ic \mid i \geq 1\}$, which is a context-free language.
\end{proof}

As an immediate corollary of Lemma~\ref{Lem: T is a CF, complete, defining system}, we have the following.

\begin{prop}
The special monoid $\Pi_2$ has context-free word problem. 
\end{prop}
\begin{proof}
Any monoid defined by a monadic complete rewriting system has context-free word problem \cite[Corollary~3.8]{Book1982}, and by Lemma~\ref{Lem: T is a CF, complete, defining system} the system $\mathcal{T}_2$ is a special (and thus monadic) complete rewriting system defining $\Pi_2$. 
\end{proof}

\begin{lem}\label{Lem: The factorisation is (abc)(abc)}
For all $i \geq 1$, the factorisation of the word $(ab^ic)^2$ into minimal invertible factors in $\Pi_2$ is $(ab^ic)(ab^ic)$. 
\end{lem}
\begin{proof}
Let $i \geq 1$. As $ab^ic$ is invertible in $\Pi_2$, any factorisation of $(ab^ic)^2$ into minimal invertible factors must obviously be a refinement of the factorisation $(ab^ic)(ab^ic)$. We thus consider factorisations of $ab^ic$. Suppose $ab^ic \equiv \lambda_1 \lambda_2 \cdots \lambda_n$ where the $\lambda_j \: (1 \leq j \leq n)$ are minimal invertible factors. If $n>2$, then there exists $1 < j < n$ such that $\lambda_j \equiv b^\mu$ for some $1 \leq \mu \leq i$. As $\lambda_j$ is minimal, it is not a proper power of a word. Thus $\mu=1$ and $b$ is invertible. Hence, as $\lambda_1 \equiv ab^\nu$ for some $\nu \geq 0$, it follows from the minimality of $\lambda_1$ that $\nu = 0$. Thus $\lambda_1 \equiv a$ is invertible. But $a$ is not invertible in $\Pi_2$; indeed, if it were, then it would be left invertible, so there would be some $u \in A^\ast$ such that $ua = 1$ in $\Pi_2$. But no rule of the rewriting system $\mathcal{T}_2$ ends in $a$, and hence, by induction of the number of rules applied, we cannot have $ua \xra{\mathcal{T}_2} 1$. As $\mathcal{T}_2$ is complete and defines $\Pi_2$, we therefore have $ua \neq 1$ in $\Pi_2$, a contradiction.

Thus $n \leq 2$. If $n = 2$, then $ab^ic \equiv \lambda_1 \lambda_2$, and consequently $\lambda_1 \equiv ab^\mu$ and $\lambda_2 \equiv b^\nu c$ for some $\mu, \nu \geq 1$ with $\mu + \nu = i$. But then $b$, being a prefix of the invertible $\lambda_2$ and a suffix of the invertible $\lambda_1$, is both right and left invertible, and hence invertible. Thus $\lambda_1 \equiv ab^\mu$ has a non-empty proper invertible suffix, contradicting its minimality. We conclude that $n=1$, and thus the factorisation of $ab^ic$ into minimal invertible factors is $(ab^ic)$, whence the factorisation of $(ab^ic)^2$.
\end{proof}

Let now $\Lambda = \{ ab^ic \mid i \geq 1\}$. By Lemma~\ref{Lem: The factorisation is (abc)(abc)}, $\Lambda$ is the set of minimal words associated to $\Pi_2$. The following proof follows the proofs of \cite[Lemmas~3.3 \& 3.4]{Zhang1992} (which consider only finitely presented special monoids) but does not rely on the technical machinery preceding it; instead, we can directly use the rewriting system $\mathcal{T}_2$.

\begin{lem}\label{Lem: Any irr invertible word is in Lambda*}
Let $u \in A^\ast$ be $\mathcal{T}_2$-irreducible and invertible. Then $u \in \Lambda^\ast$. 
\end{lem}
\begin{proof}
We will first prove that if $w \in A^\ast$ is $\mathcal{T}_2$-irreducible and \textit{right} invertible, then $w \in \Pre(\Lambda)^\ast$. The proof is by induction on $|w|$. If $|w|=0$, then the claim is obvious. Suppose for induction that the claim is true for all right invertible words of length less than $|w|$, and that $|w|>0$. As $w$ is right invertible, there exists some $w' \in A^\ast$ such that $ww' =_{\Pi_2} 1$. Assume without loss of generality that $w'$ is $\mathcal{T}_2$-irreducible. As $\mathcal{T}_2$ is complete and defines $\Pi_2$ by Lemma~\ref{Lem: T is a CF, complete, defining system}, we have $ww' \xra{\mathcal{T}_2} 1$. As $w$ and $w'$ are $\mathcal{T}_2$-irreducible, but $ww'$ is not as $|ww'| > 0$, we can hence write $w$ and $w'$ as $w \equiv w_0 w_1$ and $w' \equiv w_0' w_1'$, with $|w_1|>0$ and $|w_0'|>0$, such that
\[
ww' \equiv w_0 w_1 w_0' w_1' \xr{\mathcal{T}_2} w_0 w_1' \xra{\mathcal{T}_2} 1,
\]
and where $w_1w_0' \equiv (ab^ic)^2$ for some $i \geq 1$. As $w_1$ is a prefix of the word $w_1 w_0' \equiv (ab^ic)^2 \in \Lambda^\ast$, we have $w_1 \in \Pre(\Lambda)^\ast$. As $|w_1|>0$, we have $|w_0|<|w|$. As $w_0$ is a prefix of the right invertible word and $\mathcal{T}_2$-irreducible word $w$, it is itself right invertible and $\mathcal{T}_2$-irreducible, and so by the inductive hypothesis $w_0 \in \Pre(\Lambda)^\ast$. Thus $w \equiv w_0 w_1 \in \Pre(\Lambda)^\ast$, and we are done with the proof of the claim.

We are now ready to prove the lemma. Let now $u \in A^\ast$ be $\mathcal{T}_2$-irreducible and invertible. Then $u$ is also right invertible; so, by the claim, let $p_1, p_2, \dots, p_n \in \Pre(\Lambda)$ be such that $u \equiv p_1 p_2 \cdots p_n$. As $p_1 p_2 \cdots p_n$ is invertible, it follows that $p_n$ is left invertible, being a suffix of an invertible word. But $p_n$ is right invertible, being a prefix of a word from $\Lambda$. Thus $p_n$ is invertible. But no non-empty proper prefix of a word from $\Lambda$ is invertible; hence $p_n \in \Lambda$. As $p_n$ and $p_1 p_2 \cdots p_n$ being invertible implies that $p_1 p_2 \cdots p_{n-1}$ is invertible, we can continue the same process to conclude that $p_1, p_2, \dots, p_n \in \Lambda$. Thus $u \in \Lambda^\ast$.
\end{proof}

The following proposition now follows directly from Lemma~\ref{Lem: T is a CF, complete, defining system} and Lemma~\ref{Lem: Any irr invertible word is in Lambda*}. 

\begin{prop}\label{Prop: Lambda* = U(M)}
The submonoid $\langle \Lambda \rangle_{\Pi_2}$ is the group of units of $\Pi_2$.
\end{prop}

We note that, in the finitely presented case, the fact that the set of minimal invertible factors in general generates the group of units is \cite[Lemma~3.4]{Zhang1992} or indeed \cite[Prop~4.2]{Ivanov2001}; in the latter, it is proved for special \textit{inverse} monoids via an algorithmic graphical procedure due to J. B. Stephen \cite{Stephen1987}.\footnote{The method, which will not be used in any capacity here but which has applications also to special monoids, is known as \textit{Stephen's procedure}, and was described already for groups in 1910 by Dehn \cite{Dehn1910} in his approximating construction of the \textit{Gruppenbild}, or Cayley graph, of a group. For more details on Dehn's construction, we refer the reader to \cite[Chapter~I.5]{Chandler1982}.} 

As no word in $\Lambda$ overlaps non-trivially with any other, we have that $\Lambda$ is a code. Using this, together with the fact that all defining words of $\Pi_2$ are elements of $\Lambda^\ast$, we can easily find a presentation for the submonoid $\langle \Lambda \rangle_{\Pi_2}$ of $\Pi_2$ using the presentation of $\Pi_2$. Malheiro \cite[Theorem~2.4]{Malheiro2005} proved a general presentation theorem for what he calls ``codified'' submonoids of monoids. We will only need a very special case of this, which we now state. For ease of notation, for the statement of the theorem we temporarily forget having defined $A = \{ a, b, c\}$ and let $A$ be an arbitrary alphabet. To every code $A_0 \subseteq A^+$ we associate an alphabet $X_0$ in bijective correspondence with $A_0$ via a map $\varphi \colon A_0 \to X_0$, which we extend to an isomorphism $\varphi \colon A_0^\ast \to X_0^\ast$. 

\begin{thm}[Malheiro \cite{Malheiro2005}]\label{Thm: Malheiro's theorem 1}
Let $R = \{ (u_1, \varepsilon), (u_2, \varepsilon), \dots \}$ be a special complete rewriting system, on an alphabet $A$, which defines the monoid $M$. Let $A_0 \subseteq A^+$ be a code such that all left-hand sides of rules in $R$ are elements of $A_0^+$, and let $X_0, \varphi$ be as above. Then 
\[
\pres{Mon}{X_0}{\varphi(u_1) = 1, \varphi(u_2) = 1, \dots}
\]
is a presentation for the submonoid $\langle A_0 \rangle_{M}$.
\end{thm}

Redefining $A = \{ a, b, c\}$, we can now apply the above theorem to our monoid.

\begin{prop}
The group $U(\Pi_2)$ is isomorphic to the free product of infinitely many finite cyclic groups $\pres{Gp}{x_i \: (i \geq 1)}{x_i^2 = 1 \: (i \geq 1)}$. In particular, $U(\Pi_2)$ is non-finitely generated.
\end{prop}
\begin{proof}
As $\Lambda$ is a code, let $X = \{ x_1, x_2, \dots, \}$ be an alphabet in bijective correspondence with $\Lambda$ via the map induced by $ab^ic \mapsto x_i$, and let $\varphi \colon \Lambda^\ast \to X^\ast$ be the homomorphism extending this map. As all left-hand sides of rules in $\mathcal{T}_2$ are words in $\Lambda^+$, it follows by Lemma~\ref{Lem: T is a CF, complete, defining system} and Theorem~\ref{Thm: Malheiro's theorem 1} that 
\begin{align*}
\langle \Lambda \rangle_{\Pi_2} &\cong \pres{Mon}{X}{\varphi((ab^ic)^2) = 1} \\
&\cong \pres{Mon}{x_i \: (i \geq 1)}{x_i^2 = 1} \cong \pres{Gp}{x_i \: (i \geq 1)}{x_i^2 = 1},
\end{align*}
where the last isomorphism is due to the fact that the $x_i$ are all clearly invertible in the monoid $\pres{Mon}{x_i \: (i \geq 1)}{x_i^2 = 1}$. By Proposition~\ref{Prop: Lambda* = U(M)}, we have $\langle \Lambda \rangle_{\Pi_2} = U(\Pi_2)$, and the claim follows. 
\end{proof}

Thus we have proved the following.

\begin{thm}
The special monoid $\Pi_2$ is a context-free monoid with non-finitely generated group of units. 
\end{thm}

As all maximal subgroups of a special monoid are isomorphic to its group of units \cite[Theorem~4.6]{Malheiro2005}, we have thus proved the main theorem of this note.

\begin{cor}
The special monoid $\Pi_2$ is a context-free monoid, none of whose maximal subgroups is finitely generated.
\end{cor}

We remark that $U(\Pi_2)$ is not context-free (as it is not finitely generated), but it is \textit{locally} context-free, i.e. every finitely generated subgroup of $U(\Pi_2)$ is context-free. This follows from the Kurosh subgroup theorem, and the fact that any finitely generated subgroup of a context-free group is itself context-free. 

One can regard $\Pi_2$ as the result of a certain procedure: start with a family of pairwise isomorphic groups $C_{2,i} = \pres{Mon}{x_i}{x_i^2 = 1}$, each of which admits a context-free complete monadic rewriting system. Then, from a suitably chosen biprefix code (in this case $\{ ab^ic \mid i \geq 1\}$), create a presentation for $\Pi_2$ by taking as generators $a, b, c$, and as defining relations the union over $i$ of all defining relations of $C_{2,i}$ subjected to the substitution $x_i \mapsto ab^ic$. Of course, more generally, we can define $\Pi_n = \pres{Mon}{a,b,c}{(ab^ic)^n = 1 \: (i \geq 1)}$ in the same manner. It is not hard to show that the earlier analysis carries over, \textit{mutatis mutandis}, to show that $U(\Pi_n)$ is a free product of infinitely many finite cyclic groups of order $n$.

However, unlike in the case $n=2$, for $n > 3$ the language $\{ (ab^ic)^n \mid i \geq 1\}$ is \textit{not} context-free; this can be proved using the pumping lemma for context-free languages, and mirrors the fact that $\{ x_1^i x_2^i \cdots x_n^i \mid i \geq 1\}$ is a context-free language if and only if $n \leq 2$. Indeed, using a rewriting system $\mathcal{T}_n$ defined in the same manner as $\mathcal{T}_2$, it is clear that $\WP_{A}^{\Pi_n} \cap (ab^\ast c)^n \#$ is the reversal of the language $\{ (ab^ic)^n \mid i \geq 1 \}$, and hence $\WP_{A}^{\Pi_n}$ is not context-free when $n>2$, as the context-free languages are closed under intersection with regular languages and reversal. Thus the special monoid $\Pi_n$ has context-free word problem if and only if $n \leq 2$. 

The above construction is quite general. We provide one potential broader usage of it. Let $G_{n}$ be the elementary abelian $2$-group $C_2^n$ of order $m = 2^n$, and let $x_1, \dots, x_{m-1}$ be its non-trivial elements. Then $G_n$ admits a finite presentation via the multiplication table of $G_n$ as
\[
G_n = \pres{Gp}{x_1, \dots, x_{m-1}}{\mathcal{R}_n} = \pres{Mon}{x_1, \dots, x_{m-1}}{\mathcal{R}_n}
\]
where the relations $\mathcal{R}_n$ are all of the form $x_i x_j = x_k$ or $x_i x_j = 1$ where $(1 \leq i \leq j \leq k \leq m-1)$. Note that the equality between the group resp. the monoid presentation is due to the presence of relations of the form $x_i^2 = 1$ for all $1 \leq i \leq m-1$, so every $x_i$ is invertible in the monoid presentation. Let $b_1, \dots, b_{m-1}$ be new symbols. Analogously to how $\Pi_n$ was obtained from $C_n$, we define a new monoid $M_n$ by, for every $i \geq 1$, rewriting the relations in $\mathcal{R}_n$ to a set of relations $\mathcal{R}'_{i,n}$ over the alphabet $\{ a, b_1, b_2, \dots, b_{m-1}, c\}$ by performing the substitution $x_{j} \mapsto ab_j^ic$, and then defining
\[
M_n = \pres{Mon}{a, b_1, \dots, b_n, c}{\bigcup_{i \geq 1} \mathcal{R}_{i,n}'}.
\]
One may show, in the same manner as before, that the rewriting system $\bigcup_{i \geq 1} \mathcal{R}_{i,n}'$ is complete (as $\mathcal{R}_n$ is complete), and that $U(M_n)$ is a free product of countably many copies of the elementary abelian $2$-group $G_n$, and hence is not context-free. While the rewriting system is \textit{not} monadic when $n>2$, its behaviour appears to the author as rather similar to that of a context-free monadic rewriting system. We therefore suspect that $M_n$ has context-free word problem, but do not have a proof.

We end this note with a connection to coherent groups. That the group of units of $\Pi_2$ is a free product of countably many finitely generated virtually free groups has many structural implications. Indeed, it follows directly from the Kurosh subgroup theorem that, in an infinite free product of finitely presented virtually free groups, any finitely generated subgroup is finitely presented. In particular, $U(\Pi_2)$ is a \textit{coherent} group (a group $G$ is coherent if every finitely generated subgroup of $G$ is also finitely presented). Furthermore, $U(\Pi_2)$ is \textit{locally} virtually free (resp. \textit{locally} context-free), i.e. every finitely generated subgroup of $U(\Pi_2)$ is virtually free (resp. context-free). This raises the following question of how bad the structure of the group of units of a context-free monoid can be. 

\begin{question}
Does there exist a context-free monoid with an incoherent maximal subgroup? Does there exist such a monoid which is special?
\end{question}

This question can also be posed with \textit{incoherent} replaced by \textit{not locally context-free}. Coherence of groups, particularly that of one-relator groups, has recently seen a good deal of study (see especially \cite{Wise2005,Baumslag2019,Louder2020}), and recent work by Gray \& Ru\v{s}kuc \cite{Gray2021} has linked the coherence of one-relator groups to the question of groups of units of special \textit{inverse} monoids. This makes the question of coherence particularly interesting. As any \textit{finitely presented} context-free special monoid necessarily has context-free group of units (see \cite{NybergBrodda2020b}), any special monoid witnessing a positive answer to the question is necessarily non-finitely presentable.

\bibliography{contextfreeunits} 
\bibliographystyle{amsplain}

\end{document}